\documentclass[reqno]{amsart}
\usepackage{amssymb,latexsym}
\usepackage{amsfonts,mathrsfs}
\usepackage[margin=1.4in]{geometry}
\usepackage[all]{xy}

\newtheorem{theorem}{Theorem}[section]
\newtheorem{lemma}[theorem]{Lemma}
\newtheorem{proposition}[theorem]{Proposition}

\theoremstyle{definition}\newtheorem{clm}[theorem]{Claim}

\newcommand{\ba}{\begin{array}}
\newcommand{\ea}{\end{array}}

\def \qed{\cqfd}

\def \b{\beta}

\def \psh{\it PSH(\O )}

\def \O{\Omega}

\def \co{\it\mathbb C^n}
\def \d{\delta}

\def \bj{\bar{j}}

\def\al{\alpha}

\def \C{\mathbb C}

\DeclareMathOperator \la {{\it\lambda}}
\DeclareMathOperator \dist {dist}

\def\qed{\vbox{\hrule
\hbox{\vrule\hbox to 5pt{\vbox to 8pt{\vfil}\hfil}\vrule}\hrule}}

\newcommand{\beg}{\begin{eqnarray*}}
\newcommand{\begn}{\begin{eqnarray}}
\newcommand{\en}{\end{eqnarray*}}
\newcommand{\enn}{\end{eqnarray}}

\begin{document}
\title{A $\mathcal C^{2,\alpha}$ estimate of the complex Monge-Amp\`ere equation}
\author{Chao Li}
\address{School of Mathematical Sciences\\
University of Science and Technology of China\\
Hefei, 230026,P.R. China\\ } \email{leecryst@mail.ustc.edu.cn}
\author{Jiayu Li}
\address{Key Laboratory of Wu Wen-Tsun Mathematics\\ Chinese Academy of Sciences\\School of Mathematical Sciences\\
University of Science and Technology of China\\
Hefei, 230026\\ and AMSS, CAS, Beijing, 100080, P.R. China\\} \email{jiayuli@ustc.edu.cn}
\author{Xi Zhang}
\address{Key Laboratory of Wu Wen-Tsun Mathematics\\ Chinese Academy of Sciences\\School of Mathematical Sciences\\
University of Science and Technology of China\\
Hefei, 230026,P.R. China\\ } \email{mathzx@ustc.edu.cn}
\thanks{The authors were supported in part by NSF in
China,  No.11625106, 11571332, 11526212.}

\begin{abstract} In this paper,  we prove a  $\mathcal C^{2,\alpha}$-estimate for the solution  to the complex Monge-Amp\`ere equation
 $\det(u_{i\bar{j}})=f$ with
 $0< f\in \mathcal C^{\alpha}$, under the assumption that $u\in \mathcal C^{1,\beta }$ for some $\beta <1$ which depends on $n$ and $\alpha $.
\end{abstract}

\maketitle

\section{Introduction}
\setcounter{equation}{0}

\hspace{0.4cm}

The complex Monge-Amp\`ere equation has been the subject of
intensive studies in the past forty years,   because its significant applications in complex analysis and complex geometry.
In order to prove  the Calabi's conjecture (\cite{Cal0}),  Yau (\cite{Y}) solved a complex Monge-Amp\`ere equation on K\"ahler manifolds.
Constructions of K\"ahler-Einstein metrics  by means of solving appropriate complex Monge-Amp\`ere
equations under
suitable smoothness assumptions were carried out by many people (see e.g. \cite{Au,CY,MY,Kob,TY1,TY2,TY3,TY4,Ti1,Ti2}). In \cite{CKNS}, Caffarelli, Kohn, Nirenberg, and Spruck solved the Dirichlet problem of the complex Monge-Amp\`ere equation with smooth data on a smooth, bounded, strongly pseudoconvex domain in $\C^{n}$. Guan (\cite{Gb}) generalized their result to the case that there is a suitable subsolution.
To overcome
regularity problem for the homogeneous complex Monge-Amp\`ere equation arising
from the Chern-Levine-Nirenberg intrinsic norms in \cite{CLN}, Bedford-Taylor (\cite{BT,BT1,BT2}) developed
the theory of weak solutions.
Kolodziej (\cite{Ko,Ko1})
proved the existence and H\"older estimate of  solution to the complex Monge-Amp\`ere
equation when the right hand side is a  nonnegative $L^{p}$ function for $p>1$.
There are further  existence and regularity results on the complex Monge-Amp\`ere equation with right hand side less regular or  degenerate,
 see references \cite{Bl1,Ch,GPF,Bl2,ZZ,TZ,He,Di,DP,EGZ} for details.

\medskip

In this paper we investigate the interior $\mathcal{C}^{2, \alpha }$-estimates for solutions to the  complex Monge-Amp\`ere equation with $\mathcal{C}^{\alpha }$ strictly positive right hand side.
For the case of the real Monge-Amp\`ere equation, sharp interior $\mathcal{C}^{2, \alpha }$-estimates were established by Caffarelli (\cite{Caf})  and sharp boundary estimates
were subsequently obtained by Trudinger and Wang (\cite{TW}).
Their arguments rely essentially  on tools in convex analysis which are not available in the complex setting.
This
accounts for  significant difference between the theories of real and
complex Monge-Amp\`ere equation.

Let $\O$ be a domain in $\co$ and $u$ be a solution of the complex Monge-Amp\`ere
equation
\begin{equation}\label{CMA}
\det(u_{i\bar{j}})=f,\end{equation}
on $\O$, where $f$ is a positive function.
 Yau (\cite{Y}) used
Calabi's $\mathcal C^{3}$-estimate  to establish the $\mathcal{C}^{2, \alpha }$ estimate, which depends on third derivatives of
$f$.  Siu (\cite{Si}) used the Evans-Krylov approach (\cite{Ev,Kr}) to get the $\mathcal{C}^{2, \alpha }$ estimate which depends on
second derivatives of $f$ (also \cite{Tr}). In \cite{Ti0}, Tian established  a new $\mathcal{C}^{2, \alpha }$ estimate which depends on the H\"older norm of $f $ and lower bound of $\Delta f $. If $f$ belongs to $\mathcal{C}^{0,1}$
or even to some Sobolev space $W^{1,p}$ for $p>2n$, there are also related estimates due to Blocki (\cite{Bl1}) and Chen-He (\cite{CH}).  In \cite{DZhZh}, Dinew, Zhang and the third-named author obtained  a $\mathcal{C}^{2, \alpha }$ estimate depending on the H\"older bound of $f$ and the bound for the
real Hessian of $u$, and this estimate is optimal according to the H\"older exponent, i.e. they proved the following theorem:

\medskip

\begin{theorem}\label{thm1}{\bf (\cite{DZhZh})}
 Let $\O$ be a domain in $\co$ and $u\in\psh\cap\ \mathcal C^{1,1}(\O)$ statisfy the complex Monge-Amp\`ere
equation (\ref{CMA}) in $\O$.
Suppose additionally that $f\geq\la>0$ in $\O$ for some constant $\la$ and $f\in\mathcal{C}^{\al}(\O)$ for some $\al\in(0,1)$. Then $u\in\mathcal C^{2,\al}(\O)$. Furthermore the $\mathcal C^{2,\al}$ norm of $u$ in any relatively compact subset is estimable in terms of $n, \al, \la, \|u\|_{\mathcal C^{1,1}(\O)}, \|f\|_{\mathcal C^{\al}(\O)}$ and the distance of the set to $\partial \O$.
\end{theorem}

\medskip

In \cite{WY}, Wang used the Krylov's technique (\cite{Kr}) to reduce the complex Monge-Amp\`ere equation  to a real Bellman-type equation, and then he applied Caffarelli's estimates (\cite{Caf0}) and the above estimate to obtain a $\mathcal{C}^{2, \alpha }$ estimate depending on $\|\Delta u\|_{L^{\infty}(\Omega )}$ instead of $\|u\|_{\mathcal C^{1,1}(\O)}$. Recently, Chen-Wang (\cite{ChW}) presented a new proof to this estimate.  Tosatti-Wang-Weinkove-Yang (\cite{TWWY}) and Chu (\cite{Chu}) extended this estimate to some other fully nonlinear elliptic equations in complex geometry.
 It remains an interesting problem to see whether  the condition  can be weakened to  $u\in\mathcal C^{1,\gamma}(\O)$ for $1>\gamma >1-2/n$ (Pogorelov type examples in \cite{Bl} show that this exponent would be optimal). In this paper, we weaken the condition to $u\in \mathcal C^{1,\beta }$ for some $\beta <1$, this  partially answers the above  problem.
Below we state our result:

\medskip

\begin{theorem}\label{thm2}
 Let $\O$ be a domain in $\co$ and $u\in\psh\cap\ \mathcal C(\O)$ be a weak solution of the complex Monge-Amp\`ere
equation (\ref{CMA}) in $\O$
with $0<\lambda \leq f \in\mathcal C^{\al}(\O) $ for some constant $\la$ and  some $\al\in(0,1)$.
If  $u \in \mathcal{C}^{1,\beta}(\O)$ with $\beta \in (\beta_0,1)$, where
\begin{equation}\label{b0}\beta_0=\beta_0(n,\alpha)=1-\frac{\alpha}{n(2+\alpha)-1}.\end{equation}
Then $u\in\mathcal C^{2,\al}(\O)$. Furthermore the $\mathcal C^{2,\al}$ norm of $u$ in any relatively compact subset is estimable in terms of $n, \al, \beta, \la, ||u||_{\mathcal C^{1, \beta }(\O)},  ||f||_{\mathcal C^{\al}(\O)}$ and the distance of the set to $\partial\O$.
\end{theorem}

\medskip

We now give an overview of our proof. In \cite{DZhZh}, the authors proved Theorem \ref{thm1} based on a method due to Wang (\cite{Wa1}) and the Bedford-Taylor interior $\mathcal C^{1,1}$ estimate (\cite{BT}) which is the sole reason why they need the assumption that $u\in\mathcal C^{1,1}$. In this paper, for any point $x_{0}\in \O$, we consider the following Dirichlet problem:
\begin{equation}\label{vp1}
\left\{
\begin{split}
& \det((\varphi_{t, x_{0}})_{i\bar{j}})=f(x_0),  \quad &\textrm{in } \quad B_{dt}(x_0),\\
& \varphi_{t, x_{0}} =\rho_{dt^{\mu}}\!*\! u ,   \quad &\textrm{on } \quad \partial B_{dt}(x_0),
\end{split}
\right.
\end{equation}
where $\rho $ is a nonnegative symmetric mollifier, constants $d \leq \textrm{dist}(x_0,\partial\Omega)$, $0<t\leq\frac{1}{2}$ and $\mu >1$.
The results in \cite{CKNS} guarantee  the existence of smooth solution $\varphi_{t, x_{0}}$ to the above Dirichlet problem. By the Bedford-Taylor interior $\mathcal C^{1,1}$ estimate, we analysis the asymptotic behavior of $\varphi_{t, x_{0}}$ as $t\rightarrow 0$. Furthermore, we can obtain a $\mathcal{C}^{1, 1}$-estimate for the weak solution $u$ of (\ref{CMA}) under the assumption that  $u \in \mathcal C^{1,\beta}$ with $\beta \in (\beta_0,1)$, where the constant $\beta_{0}$ is given in (\ref{b0}), see Theorem \ref{mthm} for details. Then Theorem \ref{thm1} implies Theorem \ref{thm2}.
We hope that our approach can be helpful to solve the  optimal $\mathcal C^{2,\alpha}$-estimate of the  complex Monge-Amp\`ere equation.

\hspace{0.3cm}

\section{Preliminary}
\setcounter{equation}{0}

For convenience, we first introduce some norms and semi-norms on certain function spaces, which can be found in \cite{GT}.\\

Let $\Omega \subset \mathbb{R}^n$ be a bounded open domain. For any $x,y\in\Omega$, we set
\begin{equation*}d_x=\dist(x,\partial\Omega), \qquad d_{x,y}=\min\{d_x,d_y\}.\end{equation*}
For any $u \in  \mathcal{C}^{k,\alpha}(\Omega)$ ($\alpha\in [0,1] $), we define the following quantities:
\begin{equation}\begin{aligned}
&|u|_{0;\Omega}=\sup\limits_{x\in \Omega}|u(x)|,\\
&[u]_{k,0;\Omega}=[u]_{k;\Omega}=\sup\limits_{x\in \Omega}|\nabla^k u(x)|,\\
&[u]_{k,\alpha;x}=\sup\limits_{\begin{subarray}{c} y\in \Omega\\ y\neq x\end{subarray}}\frac{|\nabla^k u(y)-\nabla^k u(x)|}{|y-x|^{\alpha}},\\
&[u]_{k,\alpha;\Omega}=\sup\limits_{x\in\Omega}[u]_{k,\alpha;x}=\sup\limits_{\begin{subarray}{c}x,y\in \Omega\\x\neq y\end{subarray}}\frac{|\nabla^k u(x)-\nabla^k u(y)|}{|x-y|^{\alpha}},\\
\end{aligned}\end{equation}
and
\begin{equation}\begin{aligned}
&[u]^*_{k,0;\Omega}=[u]^*_{k;\Omega}=\sup\limits_{x\in\Omega}d_x^k|\nabla^k\! u(x)|,\\
&|u|^*_{k,0;\Omega}=|u|^*_{k;\Omega}=\sum\limits_{i=0}^{k}[u]^*_{i;\Omega},\\
&[u]^*_{k,\alpha;\Omega}=\sup\limits_{\begin{subarray}{c}x,y\in\Omega\\x\neq y\end{subarray}}d_{x,y}^{k+\alpha} \frac{|\nabla^k\! u(x)-\nabla^k\! u(y)|}{|x-y|^{\alpha}},\\
&|u|^*_{k,\alpha;\Omega}=|u|^*_{k;\Omega}+[u]^*_{k,\alpha;\Omega}.
\end{aligned}\end{equation}
By the definitions, we see that
\begin{equation}|u|_{0;\Omega}=[u]_{0;\Omega}=|u|^*_{0;\Omega},\end{equation}
and
\begin{equation}[u]^*_{k-1,1;\Omega}\leq C(n,k)([u]^*_{k-1;\Omega}+[u]^*_{k;\Omega}),\qquad [u]^*_{k;\Omega}\leq C(n,k)[u]^*_{k-1,1;\Omega},\end{equation}
for $k\geq 1$.
For any $f\in \mathcal{C}^{\alpha}(\Omega)$, we define
\begin{equation}|f|^{(k)}_{0,\alpha;\Omega}=\sup\limits_{x\in\Omega}d_x^k|f(x)|+\sup\limits_{\begin{subarray}{c}x,y\in\Omega\\x\neq y\end{subarray}}d_{x,y}^{k+\alpha} \frac{|f(x)-f(y)|}{|x-y|^{\alpha}}.\end{equation}

\medskip

In the following,  we list some  results we need including: the $\mathcal{C}^{2,\alpha}$ estimate for the Laplacian equation, the interpolation inequalities for H\"older continuous functions, the Bedford-Taylor interior $\mathcal{C}^{1,1}(\O)$ estimate and the Calabi interior $\mathcal{C}^{3}$ estimate for solutions of the complex Monge-Amp\`ere equation.

\medskip

\begin{proposition}[Theorem 4.8 of \cite{GT}]\label{sch}
Let $u\in \mathcal{C}^{2}(\Omega)$ and $f\in\mathcal{C}^\alpha(\Omega)(\alpha\in (0,1))$ satisfy the Laplacian equation on $\Omega$
\begin{equation*}\Delta u=f,\end{equation*}
then
\begin{equation}[u]^*_{2,\alpha;\Omega}\leq C(n,\alpha)(|u|_{0;\Omega}+|f|^{(2)}_{0,\alpha;\Omega}).\end{equation}
\end{proposition}

\medskip

\begin{proposition}[Lemma 6.32 of \cite{GT}]\label{inter}
Suppose $j+\beta<k+\alpha$, where $j,k=0,1,2,\cdots$ and $\alpha,\beta\in[0,1]$. Let $u\in \mathcal{C}^{k,\alpha}(\Omega)$, then for any $\mu\in(0, 1]$
\begin{equation}[u]^*_{j,\beta;\Omega}\leq C(n,k)(\mu^{-(j+\beta)}|u|_{0;\Omega}+\mu^{k+\alpha-(j+\beta)}[u]^*_{k,\alpha;\Omega}).\end{equation}
\end{proposition}

\medskip

\begin{proposition}[Theorem 6.7 of \cite{BT}]\label{c2}
Let $B=B_r(x_0)\subset\mathbb{C}^n$, and  $u \in PSH(B)\cap\mathcal{C}(\bar{B})$ be a weak solution of the complex Monge-Amp\`ere equation
\begin{equation*}\left\{\begin{array}{ll}\det(u_{i\bj})=f,& \textrm{in }B,\\
u=\varphi,&\textrm{on }\partial B,
\end{array}
\right.
\end{equation*}
where $f\geq 0$, $f^\frac{1}{n}\in \mathcal{C}^{1,1}(\bar{B})$ and $\varphi\in\mathcal{C}^{1,1}(\partial B)$.
Then $u\in \mathcal{C}^{1,1}(B)$ and for any $t\in(0,1)$
\begin{equation}[u]_{1,1;B_{(1-t)r}(x_0)}\leq C(n,t)([\varphi]_{1,1;\partial B}+|f^{\frac{1}{n}}|_{0;B}+r[f^{\frac{1}{n}}]_{1;B}+r^2[f^{\frac{1}{n}}]_{1,1;B}).\end{equation}
\end{proposition}

\medskip

\begin{proposition}[Theorem 1 of \cite{SR}]\label{c3}
Let $\Omega\subset\mathbb{C}^n$ be a bounded domain, and let $u \in PSH(\Omega)\cap\mathcal{C}^5({\Omega})$ be a solution of complex Monge-Amp\`ere equation
\begin{equation*}\det(u_{i\bj})=f,\end{equation*}
in $\O$, where $0<m\leq f\in\mathcal{C}^3(\Omega)$ for some constant $m$.
If $\Delta u\leq K$ on $\Omega$ and $[\log f]^*_{2;\Omega}+[\log f]^*_{3;\Omega}\leq L$, then
\begin{equation}\label{c3ineq}[u_{i\bar{j}}]^*_{1;\Omega}\leq C(n)K^{n+1}m^{-1}(1+L^{\frac{1}{2}}).\end{equation}
\end{proposition}

\medskip

The above  proposition was essentially proved in \cite{SR}. But  no exact formula like (\ref{c3ineq}) is given in the origin theorem of \cite{SR}, so we will give a proof   in the appendix for  readers' convenience.

\hspace{0.3cm}

\section{Main theorem and the key lemma}
\setcounter{equation}{0}

By Theorem \ref{thm1}, we know that to prove Theorem \ref{thm2} is reduced to obtain $\mathcal{C}^{1, 1}$-estimates, i.e. to show the following theorem:

\medskip

\begin{theorem}[Main Theorem]\label{mthm}
Let  $u\in\psh\cap\ \mathcal C(\O)$ be a weak solution of the complex Monge-Amp\`ere
equation (\ref{CMA}) in $\O$ and satisfy the same assumptions as that in Theorem \ref{thm2}.
Then $u\in\mathcal C^{1,1}(\O)$. Furthermore the $\mathcal C^{1,1}$ norm of $u$ in any relatively compact subset is estimable in terms of $n, \al, \beta, \la, ||u||_{\mathcal C^{1, \beta }(\O)},  ||f||_{\mathcal C^{\al}(\O)}$ and the distance of the set to $\partial\O$.
\end{theorem}

\medskip

In order to  prove Theorem \ref{mthm}, we need the following lemma:

\medskip

\begin{lemma}[Key Lemma]\label{klm}
Let  $u\in\psh\cap\ \mathcal C(\O)$ be a weak solution of the complex Monge-Amp\`ere
equation (\ref{CMA}) with $0<\lambda \leq f \in\mathcal C^{\al}(\O) $ for some constant $\lambda$ and  some $\alpha\in(0,1]$. Given any $\delta \in (\beta_0(n,\alpha),1]$, assume that $u \in \mathcal{C}^{1,\beta}(\O)$ with any $\beta \in (\phi(\delta ),1)$, where
\begin{equation}\label{phi}\phi(\delta)=\phi_{n,\alpha}(\delta)=1-\frac{2((2+\alpha)(2-\delta)-(1+\delta))}{(2+\alpha)(1+\delta)(n+1)+((2+\alpha)(2-\delta)-(1+\delta))}.\end{equation}
Then $u \in \mathcal{C}^{1,\delta}(\O )$. Furthermore the $\mathcal C^{1,\delta}$ norm of $u$ in any relatively compact subset is estimable in terms of $n, \alpha, \lambda, \beta, \delta, ||u||_{\mathcal C^{1, \beta }(\Omega)},  ||f||_{\mathcal C^{\alpha}(\Omega)}$ and the distance of the set to $\partial\Omega$.
\end{lemma}

\medskip

We first give some information about $\phi$ defined in (\ref{phi}), and then use Lemma \ref{klm} to prove Theorem \ref{mthm}.
In fact, $\b_0$ defined by (\ref{b0}) is the fixed point of $\phi$, i.e.
\begin{equation}\phi(\b_0)=\b_0.\end{equation}
Furthermore $\phi$ is a strictly increasing function on $[\b_0,1]$ and satisfies
\begin{equation}
\phi(\d)<\d,
\end{equation}
for any $\d\in(\b_0,1]$.

\medskip

\begin{proof}[Proof of Theorem \ref{mthm}] Consider the sequence $\{\delta_i\}_{i=1}^{\infty}$ defined by
\begin{equation*}
\begin{aligned}
\delta_1 &=\phi(1)<1,\\
\delta_{i+1}&=\phi(\delta_i),\qquad i=1,2,\cdots.
\end{aligned}
\end{equation*}
One can easily check that $\{\delta_i\}_{i=1}^{\infty}$ is a decreasing sequence which converges to $\beta_0$.
Since $\beta \in (\beta_0,1)$, we can find some $k \geq 1$, so that $\beta > \delta_k$. We need only to prove that Theorem \ref{mthm} is true if $\beta \in (\delta_{k}, 1)$ for all $k\geq1$.
\begin{itemize}
\item[(1)]If $k=1$, $\beta>\phi(1)$, Theorem \ref{mthm} is a direct corollary of Lemma \ref{klm}.
\item[(2)]Assuming that  Theorem \ref{mthm} is true for  some $l\geq 1$, i.e.
if the weak solution $u\in\mathcal{C}^{1,\gamma}(\Omega)$ with any $\gamma\in(\delta_l,1)$, then $u\in\mathcal{C}^{1,1}(\Omega)$.
If  $\beta>\delta_{l+1}=\phi(\delta_l)$, we can find some $\gamma\in(\delta_l,1]$, such that $\beta>\phi(\gamma)$. By Lemma \ref{klm}, we have $u\in\mathcal{C}^{1,\gamma}(\Omega)$.
Then  the assumption implies $u\in\mathcal{C}^{1,1}(\Omega)$.
To obtain the required estimate on $||u||_{\mathcal{C}^{1,1}(\overline{\Omega '})}$ for some relative compact subset $\overline{\Omega '}$ of $\Omega$, we consider the domain
\begin{equation*}\Omega''=\left\{x\in\Omega|\dist(x,\partial\Omega)>\textstyle{\frac{1}{2}}\dist(\overline{\Omega'},\partial\Omega)\right\}.\end{equation*}
We can apply Lemma \ref{klm} to $u$ on $\Omega$ to estimate $||u||_{\mathcal{C}^{1,\gamma}(\overline{\Omega ''})}$. Then use the assumption to $u$ on $\Omega''$ to estimate $||u||_{\mathcal{C}^{1,1}(\overline{\Omega '})}$. So Theorem \ref{mthm} is true for $k=l+1$.
\end{itemize}
Combining (1) and (2), we know that Theorem \ref{mthm} is true for all $k\ge 1$. This concludes the proof of Theorem \ref{mthm}.
\end{proof}

\hspace{0.3cm}

Now we move on to the proof of Lemma \ref{klm}. It is based on the proof of Theorem 1.1 of \cite{DZhZh} and  a family of auxiliary functions $\varphi_{t,x_0}$, as mentioned in the introduction. In the following, we give the exact definition of the auxiliary functions.

\medskip

Let $\rho(z)$ be a smooth radially symmetrical function on $\mathbb{C}^n$ satisfying
\begin{itemize}
\item[1).] $\rho \ge 0$ and $\textrm{supp}(\rho) \subset \{|z|<1\}$;
\item[2).] $\int_{\mathbb{C}^n}\rho=1$;
\item[3).] $| \nabla \! \rho| \leq C_1(n), \quad | \nabla^2 \! \rho| \leq C_2(n)$,
where $C_1(n)$ and $C_2(n)$ are constants depending only on $n$.
\end{itemize}
For any $\varepsilon > 0$, we set
\begin{equation}\rho_{\varepsilon}(z)=\varepsilon^{-2n}\rho \left(\frac{z}{\varepsilon}\right).\end{equation}
For any point $x_0 \in \Omega$,  $\varphi_{t,x_0}$  is defined by
\begin{equation}\label{vp2}
\left\{
\begin{array}{c}
\varphi_{t,x_0}\in PSH(B_{dt}(x_0))\cap C(\bar{B}_{dt}(x_0)),\\
\begin{array}{cl}
\det((\varphi_{t,x_0})_{i\bar{j}})=f(x_0), & \textrm{in } B_{dt}(x_0),\\
\varphi_{t,x_0}=\rho_{dt^{\mu}}\!*\! u, & \textrm{on } \partial B_{dt}(x_0),
\end{array}
\end{array}
\right.
\end{equation}
where $d \leq \textrm{dist}(x_0,\partial\Omega)$, $0<t\leq\frac{1}{2}$ and $\mu >1$.
By the results of \cite{CKNS}, we know that every $\varphi_{t,x_0}$ is smooth.

\medskip

Note that we only need to consider the case $\beta <\delta\leq 1$, and we can assume $[u]_{1,\beta;\Omega}$, $|f^{\frac{1}{n}}|_{0;\Omega}$ and $[f^{\frac{1}{n}}]_{0,\alpha;\Omega}$ are finite.
For convenience, we use $C$ to denote constants depending only on $n$, $\alpha$, $\beta$, $\delta$, $[u]_{1,\beta;\Omega}$, $\lambda$, $|f^{\frac{1}{n}}|_{0;\Omega}$, $[f^{\frac{1}{n}}]_{0,\alpha;\Omega}$ and $d$. We also use $C(A,\cdots)$ to denote constants  depending on additional $A,\cdots$.

\medskip

From the definition, it is easy to check that
\begin{equation}u\leq \rho_{dt^{\mu}}\!*\! u \leq u+Ct^{\mu (1+\beta)},\end{equation}
and
\begin{equation}\label{C2*}|\nabla^2 (\rho_{dt^{\mu}}\!*\! u)|\leq Ct^{\mu(\beta-1)}.\end{equation}
By Theorem 5.7 of \cite{BT}, we know that
\begin{equation}(\det((\rho_{dt^{\mu}}\!*\!u)_{i\bar{j}}))^{\frac{1}{n}}\geq \rho_{dt^{\mu}}\!*\!f^{\frac{1}{n}}
\geq f(x_0)^{\frac{1}{n}}-Ct^{\alpha}.\end{equation}
Using the concavity of $\{\det (\cdot )\}^{\frac{1}{n}}$, for positive defined Hermitian matrices A and B, we have
\begin{equation*}(\det(A+tB))^{\frac{1}{n}} \geq (\det(A))^{\frac{1}{n}}+t(\det(B))^{\frac{1}{n}}.\end{equation*}
By this inequality, the definition of $\varphi_{t,x_0}$ and the known estimates on $\rho_{dt^{\mu}}\!*\!u$, we can find some sufficiently large constant $C$ such that
\begin{equation*}\left\{ \begin{split}
&(\det((\rho_{dt^{\mu}}\!*\!u+Ct^{\alpha}(|z-x_0|^2-d^2t^2))_{i\bar{j}}))^{\frac{1}{n}} \geq (\det((\varphi_{t,x_0})_{i\bar{j}}))^{\frac{1}{n}}, & \textrm{in } B_{dt}(x_0),\\
&\rho_{dt^{\mu}}\!*\!u=\varphi_{t,x_0}, &\textrm{on } \partial B_{dt}(x_0),
\end{split}
\right.
\end{equation*}
and
\begin{equation*}\left\{ \begin{split}
&(\det(u_{i\bar{j}}))^{\frac{1}{n}} \leq (\det((\varphi_{t,x_0}+Ct^{\alpha}(|z-x_0|^2-d^2t^2))_{i\bar{j}}))^{\frac{1}{n}}, & \textrm{in } B_{dt}(x_0),\\
&u+Ct^{\mu(1+\beta)}\geq \varphi_{t,x_0}, &\textrm{on } \partial B_{dt}(x_0).
\end{split}
\right.
\end{equation*}
The comparison principle of continuous plurisubharmonic functions (Theorem A of \cite{BT}) implies
\begin{equation*}\rho_{dt^{\mu}}\!*\!u+Ct^{\alpha}(|z-x_0|^2-d^2t^2)\leq \varphi_{t,x_0}\leq u+Ct^{\mu(1+\beta)}-Ct^{\alpha}(|z-x_0|^2-d^2t^2).\end{equation*}
So we have
\begin{equation}\label{3.9}\sup_{x\in B_{dt}(x_0)}|\varphi_{t,x_0}-u|\leq C(t^{\mu(1+\beta)}+t^{2+\alpha}).\end{equation}
Using the interior $\mathcal C^{1,1}$ estimate due to Bedford and Taylor (Proposition \ref{c2}) and the $\mathcal C^2$ estimate (\ref{C2*}) of $\rho_{dt^{\mu}}\!*\! u$, we have the following estimate
\begin{equation}\label{3.10}
[\varphi_{t,x_0}]_{1,1;B_{\frac{1}{2}dt}(x_0)} \leq C(Ct^{\mu(\beta-1)}+f(x_0)^{\frac{1}{n}})\leq Ct^{\mu(\beta-1)}.
\end{equation}

\medskip

We have the following claim:
\begin{clm} \label{clm} By choosing proper $\mu=\mu(\beta,\delta,\alpha)$, we can obtain
\begin{itemize}
\item[(1)] \begin{equation}\lim_{t\rightarrow 0}\nabla \varphi_{t,x_0}(x_0)=\nabla u(x_0).\end{equation}
Moreover, we have
\begin{equation}|\nabla \varphi_{t,x_0}(x_0)-\nabla u(x_0)|\leq C(dt)^{\delta}.\end{equation}

\item[(2)]
\begin{equation}|\nabla \varphi_{t,x_0}(x)-\nabla \varphi_{t,x_0}(y)|\leq C|x-y|^{\delta},\end{equation}for any  $x, y \in B_{\frac{1}{4}dt}(x_0)$.

\item[(3)] Let $x_0,y_0 \in \Omega$, and $d\leq \min\{\textrm{dist}(x_0,\partial \Omega),\textrm{dist}(y_0,\partial \Omega)\}$.
If $|x_0-y_0|\leq \frac{1}{8}dt$ and $B_{\frac{1}{4}dt}(z)\subset B_{\frac{1}{2}dt}(x_0)\cap B_{\frac{1}{2}dt}(y_0)$, then we have
\begin{equation}|\nabla \varphi_{t,x_0}-\nabla \varphi_{t,y_0}|\leq C(dt)^{\delta},\end{equation}
on $B_{\frac{1}{8}dt}(z)$.
\end{itemize}
\end{clm}

\medskip

Using this claim, we can prove the key lemma.
\begin{proof}[Proof of the key lemma] Let $\overline{\Omega '}$ be a relatively compact subset of $\Omega$ and $d=\dist(\overline{\Omega '},\partial\Omega)$.\\
For any distinct $x,y\in \Omega$, we have $d\leq\min\{\dist(x,\partial \Omega),\dist(y,\partial \Omega)\}$.
\begin{itemize}
\item If $|x-y| \geq \frac{1}{16}d$, then we have
\begin{equation}
\frac{|\nabla u(x)-\nabla u(y)|}{|x-y|^{\delta}} \leq (16d^{-1})^{\delta-\beta}\frac{|\nabla u(x)-\nabla u(y)|}{|x-y|^{\beta}}\leq C.
\end{equation}
\item If $|x-y| < \frac{1}{16}d$, we set $t=8\frac{|x-y|}{d}<\frac{1}{2}$. Then
\begin{equation}
\begin{aligned}
|\nabla u(x)-\nabla u(y)| \leq & |\nabla u(x)-\nabla \varphi_{t,x}(x)|+|\nabla \varphi_{t,x}(x)-\nabla \varphi_{t,x}(y)| \\
& + |\nabla \varphi_{t,x}(y)-\nabla \varphi_{t,y}(y)|+|\nabla \varphi_{t,y}(y)-\nabla u(y)|.
\end{aligned}
\end{equation}
By (1) and (2) of Claim \ref{clm}, we have
\begin{equation*}|\nabla u(x)-\nabla \varphi_{t,x}(x)|\leq C(dt)^{\delta},\end{equation*}
\begin{equation*}|\nabla \varphi_{t,y}(y)-\nabla u(y)|\leq C(dt)^{\delta},\end{equation*}
and
\begin{equation*}|\nabla \varphi_{t,x}(x)-\nabla \varphi_{t,x}(y)|\leq C(dt)^{\delta}.\end{equation*}
Since $|x-y| = \frac{1}{8}dt$ and $B_{\frac{1}{4}dt}(y)\subset B_{\frac{1}{2}dt}(x)\cap B_{\frac{1}{2}dt}(y)$, by (3) of Claim \ref{clm}
\begin{equation*}|\nabla \varphi_{t,x}(y)-\nabla \varphi_{t,y}(y)|\leq C(dt)^{\delta}.\end{equation*}
These inequalities imply
\begin{equation}
|\nabla u(x)-\nabla u(y)| \leq C(dt)^{\delta}\leq C|x-y|^{\delta} .
\end{equation}
\end{itemize}

\medskip

We conclude that
\begin{equation}[u]_{1,\delta;\Omega'}\leq C.\end{equation}
\end{proof}

To complete the proof of Theorem \ref{mthm}, we only need to prove Claim \ref{clm}.

\hspace{0.3cm}

\section{Properties of the auxiliary functions}
\setcounter{equation}{0}

In the previous section, we have given some preliminary properties of the auxiliary functions. In this section, we will show more estimates on them and give a proof of Claim \ref{clm}.

\medskip

The inequality (\ref{3.10}) implies that
\begin{equation}\label{Lphi}
\Delta\varphi_{t,x_0} \leq Ct^{\mu(\beta-1)},\end{equation}
on $B_{\frac{1}{2}dt}(x_0)$.
Applying Proposition \ref{c3} to $\varphi_{t,x_0}$ on $B_{\frac{1}{2}dt}(x_0)$, we have
\begin{equation}
[(\varphi_{t,x_0})_{i\bj}]^*_{1;B_{\frac{1}{2}dt}(x_0)}\leq Ct^{(n+1)\mu(\beta-1)},
\end{equation}
and consequently
\begin{equation}\label{Lphi1}
[\Delta\varphi_{t,x_0}]^*_{1;B_{\frac{1}{2}dt}(x_0)}\leq Ct^{(n+1)\mu(\beta-1)}.
\end{equation}
Applying Proposition \ref{inter} to $\Delta \varphi_{t,x_0}$, we have
\begin{equation}\label{Lphig}
|\Delta\varphi_{t,x_0}|^*_{0,\gamma;B_{\frac{1}{2}dt}(x_0)}\leq Ct^{(n+1)\mu(\beta-1)},
\end{equation}
for any $\gamma \in (0,1]$.

\medskip

\begin{proof}[Proof of  Claim \ref{clm}]
For fixed $x_0 \in \Omega$, $d\leq \textrm{dist}(x_0,\partial\Omega)$ and $t\leq\frac{1}{2}$, we set
\begin{equation*}
t_k=2^{-k}t,
\qquad u_k=\varphi_{t_k,x_0},
\qquad B_k=B_{dt_k}(x_0),
\end{equation*}
and
\begin{equation}\label{defvk} v_k =u_{k-1}-u_k ,\end{equation}
where $k=0,1,2,\cdots$ and $k\geq1$ in (\ref{defvk}).
By (\ref{3.9}) and (\ref{Lphig}), for $k\geq1$ and any $\gamma\in(0,1)$, we have
\begin{align*}
&|u_{k-1}-u|_{0;B_{k+1}}\leq C(t_k^{\mu(1+\beta)}+t_k^{2+\alpha}),\\
&|u_k-u|_{0;B_{k+1}}\leq C(t_k^{\mu(1+\beta)}+t_k^{2+\alpha}),\\
&|\Delta u_{k-1}|^*_{0,\gamma;B_{k+1}} \leq Ct_k^{(n+1)\mu(\beta-1)},\\
&|\Delta u_k|^*_{0,\gamma;B_{k+1}} \leq Ct_k^{(n+1)\mu(\beta-1)},
\end{align*}
and consequently
\begin{align}
\label{vk0}&|v_k|_{0;B_{k+1}}\leq C(t_k^{\mu(1+\beta)}+t_k^{2+\alpha}),\\
\label{Dvk}&|\Delta v_k|^*_{0,\gamma;B_{k+1}} \leq Ct_k^{(n+1)\mu(\beta-1)}.
\end{align}
By the above estimates and Proposition \ref{sch}, we have the following $\mathcal{C}^{2,\gamma}$ estimates for $v_k$
\begin{equation}\label{vk2a}
[v_k]^*_{2,\gamma;B_{k+1}}\leq C(\gamma)t_k^{(n+1)\mu(\beta-1)+2}.
\end{equation}
Applying Proposition \ref{inter} to $v_k$, we have
\begin{equation}\label{vk1d}
[v_k]^*_{1,\delta;B_{k+1}}\leq C(\gamma)\left(\varepsilon^{-(1+\delta)}(t_k^{\mu(1+\beta)}+t_k^{2+\alpha})+\varepsilon^{1+\gamma-\delta}t_k^{(n+1)\mu(\beta-1)+2}\right),
\end{equation}
for any $\varepsilon\in(0,1]$.
Let
\begin{equation}
\varepsilon^{2+\gamma}=\left(t_k^{\mu(1+\beta)}+t_k^{2+\alpha}\right)\left(2t_k^{(n+1)\mu(\beta-1)+2}\right)^{-1},
\end{equation}
then we obtain
\begin{equation}\label{vk1d1}
[v_k]^*_{1,\delta;B_{k+1}}\leq C(\gamma)\left(t_k^{\frac{1+\gamma-\delta}{2+\gamma}\mu(1+\beta)}+t_k^{\frac{1+\gamma-\delta}{2+\gamma}(2+\alpha)}\right)t_k^{\frac{1+\delta}{2+\gamma}\left((n+1)\mu(\beta-1)+2\right)},
\end{equation}
which implies:
\begin{equation}\label{vk1d2}
[v_k]_{1,\delta;B_{k+2}}\leq C(\gamma)\left(t_k^{\frac{1+\gamma-\delta}{2+\gamma}\mu(1+\beta)}+t_k^{\frac{1+\gamma-\delta}{2+\gamma}(2+\alpha)}\right)
t_k^{\frac{1+\delta}{2+\gamma}\left((n+1)\mu(\beta-1)+2\right)-(1+\delta)}.
\end{equation}

\medskip

For further consideration, we want to obtain the following estimate $[v_k]_{1,\delta;B_{k+2}}\le At_k^{\epsilon}$ for some positive constants $A$ and $\epsilon$ which don't depend on $t$ or $k$.
According to (\ref{vk1d2}), we need to find some $\gamma\in(0,1)$ and  $\mu>1$ such that
\begin{equation}\label{mu0}
\left\{\begin{array}{l}
\frac{1+\gamma-\delta}{2+\gamma}\mu(1+\beta)+\frac{1+\delta}{2+\gamma}\left(\mu(\beta-1)(n+1)+2\right)-(1+\delta)>0,\\
\frac{1+\gamma-\delta}{2+\gamma}(2+\alpha)+\frac{1+\delta}{2+\gamma}\left(\mu(\beta-1)(n+1)+2\right)-(1+\delta)>0.
\end{array}
\right.
\end{equation}
By solving (\ref{mu0}) as an inequalities system of unknown $\mu$, we have
\begin{equation}\label{mu1}
\frac{\gamma(1+\delta)}{(1+\gamma-\beta)(1+\beta)-(1+\delta)(1-\beta)(n+1)}<\mu<\frac{(1+\gamma-\delta)(2+\alpha)-\gamma(1+\delta)}{(1+\delta)(1-\beta)(n+1)}.
\end{equation}
Then we find the condition that $\beta$, $\delta$ and $\gamma$ should satisfy is
\begin{equation}
\frac{\gamma(1+\delta)}{(1+\gamma-\beta)(1+\beta)-(1+\delta)(1-\beta)(n+1)}<\frac{(1+\gamma-\delta)(2+\alpha)-\gamma(1+\delta)}{(1+\delta)(1-\beta)(n+1)},
\end{equation}
equivalently
\begin{equation}\label{bd1}
\beta>1-\frac{2\left((2+\alpha)(1+\gamma-\delta)-(1+\delta)\right)}{(2+\alpha)(1+\delta)(n+1)+\left((2+\alpha)(1+\gamma-\delta)-(1+\delta)\right)}.
\end{equation}
By the conditions of Lemma \ref{klm}, we already have
\begin{equation}\label{bd0}
\beta>1-\frac{2\left((2+\alpha)(2-\delta)-(1+\delta)\right)}{(2+\alpha)(1+\delta)(n+1)+\left((2+\alpha)(2-\delta)-(1+\delta)\right)}.
\end{equation}
Using (\ref{bd0}), we can choose some $\gamma=\gamma(n,\alpha,\beta,\delta)$ sufficiently close to 1, such that (\ref{bd1}) holds.
Then we can return to (\ref{mu1}) and choose some $\mu=\mu(n,\alpha,\beta,\delta)$ such that (\ref{mu0}) holds.
We will use the chosen $\mu$ and $\gamma$ till we complete the proof of Claim \ref{clm}. It is remarkable that $\mu$ and $\gamma$ depend only on $n$, $\alpha$, $\beta$ and $\delta$.

Now that (\ref{mu0}) holds, we can find some $\epsilon=\epsilon(n,\alpha,\beta,\delta)$ such that
\begin{equation}
\left\{\begin{array}{l}
\frac{1+\gamma-\delta}{2+\gamma}\mu(1+\beta)+\frac{1+\delta}{2+\gamma}\left(\mu(\beta-1)(n+1)+2\right)-(1+\delta)>\epsilon,\\
\frac{1+\gamma-\delta}{2+\gamma}(2+\alpha)+\frac{1+\delta}{2+\gamma}\left(\mu(\beta-1)(n+1)+2\right)-(1+\delta)>\epsilon,
\end{array}
\right.
\end{equation}
and then we obtain the desired estimate
\begin{equation}\label{vk1d3}
[v_k]_{1,\delta;B_{k+2}}\le Ct_k^{\epsilon}.
\end{equation}
Moreover, by (\ref{mu0}), (\ref{vk0}) and (\ref{vk1d1}), it holds that
\begin{equation}
|v_k|_{0;B_{k+1}}\leq Ct_k^{1+\delta},\qquad
[v_k]^*_{1,\delta;B_{k+1}}\leq Ct_k^{1+\delta}.
\end{equation}
Applying Proposition \ref{inter}, we have
\begin{equation}
[v_k]^*_{1;B_{k+1}}\leq Ct_k^{1+\delta},
\end{equation}
which implies
\begin{equation}\label{vk1}
[v_k]_{1;B_{k+2}}\leq Ct_k^{\delta}.
\end{equation}

\medskip

Now we can give a $\mathcal C^{1,\delta}$ estimate for $\varphi_{s,x_0}$ on $B_{\frac{1}{4}ds}(x_0)$.
\begin{itemize}
\item[(a)] If $s\in \big(\frac{1}{4}, \frac{1}{2}\big]$, by (\ref{3.10}), we have
\begin{equation*}[\varphi_{s,x_0}]_{1,1;B_{\frac{1}{2}ds}(x_0)}\leq Cs^{\mu(\beta-1)}\leq C,\end{equation*}
and then
\begin{equation}\label{4.24}[\varphi_{s,x_0}]_{1,\delta;B_{\frac{1}{2}ds}(x_0)}\leq (ds)^{1-\delta}[\varphi_{s,x_0}]_{1,1;B_{\frac{1}{2}ds}(x_0)}\leq C.\end{equation}
\item[(b)] If $s\in (0,\frac{1}{4}]$, we can find some integer $k\geq 1$ such that
\begin{equation*}\textstyle t=2^{k}s\in \big(\frac{1}{4},\frac{1}{2}\big].\end{equation*}
Setting $B=B_{\frac{1}{4}ds}(x_0)=B_{k+2}$, we have
\begin{equation}
\begin{split}
[\varphi_{s,x_0}]_{1,\delta;B}& =[u_k]_{1,\delta;B}=\left[\varphi_{t,x_0}-\textstyle{\sum\limits_{i=1}^k} v_k\right]_{1,\delta;B}\\
& \leq [\varphi_{t,x_0}]_{1,\delta;B}+\textstyle{\sum\limits_{i=1}^k} [v_i]_{1,\delta;B}\\
&\leq C+C\textstyle{\sum\limits_{i=1}^k} t_i^{\epsilon},
\end{split}
\end{equation}
where the last inequality holds by (a) and the inequality (\ref{vk1d3}).
Since
\begin{align*}
\textstyle{\sum\limits_{i=1}^k} t_i^{\epsilon} =t^{\epsilon}\textstyle{\sum\limits_{i=1}^k} 2^{-\epsilon i}\leq t^{\epsilon}\textstyle{\sum\limits_{i=1}^{\infty}} 2^{-\epsilon i}
\leq 2^{-\epsilon}(1-2^{-\epsilon})^{-1},
\end{align*}
we again obtain that
$[\varphi_{s,x_0}]_{1,\delta;B}\leq C $.
\end{itemize}

By the discussion above,  we have
\begin{equation}\label{c2_2}
[\varphi_{s,x_0}]_{1,\delta;B_{\frac{1}{4}ds}(x_0)}\leq C ,
\end{equation}
for all $s\in\big(0,\frac{1}{2}\big]$. The estimate (\ref{c2_2}) implies
(2) of Claim \ref{clm}.\\

\medskip

Now we want to prove
\begin{equation}\label{clm11}
\lim_{t\rightarrow 0}|\nabla \varphi_{t,x_0}(x_0)-\nabla u(x_0)|=0.
\end{equation}
Let us consider the following functions
\begin{align}
W_t(x) & =\frac{\varphi_{t,x_0}(x_0+tx)-\varphi_{t,x_0}(x_0)}{t},\\
V_t(x) & =\frac{u(x_0+tx)-u(x_0)}{t},
\end{align}
on $B_{\frac{1}{4}d}(0)$.
By (\ref{mu0}), we know that
\begin{equation*}\mu(1+\beta)>1+\delta,\end{equation*}
and
\begin{equation}
|\varphi_{t,x_0}-u|\leq C(t^{\mu(1+\beta)}+t^{2+\alpha})\leq Ct^{1+\delta}.
\end{equation}
This implies
\begin{equation}|W_t-V_t|\leq Ct^{\delta}.\end{equation}
Using the Taylor expansion, the $\mathcal C^{1,\delta}$ estimate of $\varphi_{t,x_0}$ and the $\mathcal C^{1,\beta}$ regularity of $u$, we obtain
\begin{align}
|W_t(x) -\langle \nabla \varphi_{t,x_0}(x_0),x\rangle|& \leq Ct^{\delta},\\
|V_t(x) -\langle \nabla u(x_0),x\rangle|& \leq Ct^{\beta},
\end{align}
where $\langle,\rangle$ is the standard inner product on the vector space $\mathbb{R}^{2n}$.
Then we have
\begin{equation}|\langle \nabla \varphi_{t,x_0}(x_0),x\rangle-\langle \nabla u(x_0),x\rangle|\leq Ct^{\beta},\end{equation}
on whole $B_{\frac{1}{4}d}(0)$, and
\begin{equation}|\nabla \varphi_{t,x_0}(x_0)-\nabla u(x_0)|\leq Ct^{\beta}.\end{equation}
This implies (\ref{clm11}).\\

Using (\ref{clm11}) and (\ref{vk1}) we have
\begin{align*}
|\nabla \varphi_{t,x_0}(x_0)-\nabla u(x_0)|& =\left| \textstyle{\sum\limits_{k=0}^{\infty}}(\nabla u_k-\nabla u_{k+1})\right|=\left| \textstyle{\sum\limits_{k=1}^{\infty}}\nabla v_k\right|\\
&\leq \textstyle{\sum\limits_{k=1}^{\infty}}|\nabla v_k|\leq C\textstyle{\sum\limits_{k=1}^{\infty}}t_k^{\delta},
\end{align*}
where
\begin{equation*}\textstyle{\sum\limits_{k=1}^{\infty}}t_k^{\delta}=t^{\delta}\textstyle{\sum\limits_{k=1}^{\infty}}2^{-k\delta}=t^{\delta}\frac{2^{-\delta}}{1-2^{-\delta}}.\end{equation*}
Then we have
\begin{equation}|\nabla \varphi_{t,x_0}(x_0)-\nabla u(x_0)| \leq Ct^{\delta}\leq C(dt)^{\delta}.\end{equation}
This completes the proof of (1) of Claim \ref{clm}.

\medskip

(3) of Claim \ref{clm} is essentially an estimate on $[v]_{1;B_{\frac{1}{8}dt}(z)}$ with $v=\varphi_{t,x_0}-\varphi_{t,y_0}$, which can be obtained by the same method that was used to estimate $[v_k]_{1;B_{k+2}}$. We omit the proof.

\end{proof}

\medskip

\hspace{0.3cm}

\section{Some further discussion}
\setcounter{equation}{0}

Let $\Omega\subset \mathbb{C}^n$ be a bounded open domain,  $f\in \mathcal{C}(\Omega)$ be a positive function and $u\in \mathcal{C}^{1,\beta}(\Omega)\cap PSH(\Omega)$ be a weak solution of the complex Monge-Amp\`ere
equation (\ref{CMA}).
In Theorem \ref{mthm}, assuming $f\in \mathcal C^{\alpha}(\Omega)$, we used auxiliary functions $\varphi_{t,x_0}$ which solve (\ref{vp1}), i.e. they satisfy
\begin{equation*}
\left\{
\begin{array}{cl}
\det((\varphi_{t,x_0})_{i\bar{j}})=f(x_0), & \textrm{in } B_{dt}(x_0),\\
\varphi_{t,x_0}=\rho_{dt^{\mu}}\!*\! u, & \textrm{on } \partial B_{dt}(x_0),
\end{array}
\right.
\end{equation*}
and proved that if $\beta>\beta_0(n,\alpha)$, then $u$ is in $\mathcal{C}^{1,1}(\Omega)$. In this section, we will show that
if $f$ has better regularity, for example $f \in \mathcal{C}^{2,1}(\Omega)$, then we can refine the auxiliary functions and then the assumption on $\beta$ can be weakened.

\medskip

If $f\in\mathcal{C}^{k,\alpha}(\Omega)$ for $k=1,2$ and $\alpha \in (0,1]$, we can use the following auxiliary functions $\varphi_{t,x_0}$ which satisfy
\begin{equation}\label{fka}
\left\{
\begin{array}{cl}
\det((\varphi_{t,x_0})_{i\bar{j}})=f_{k,x_0}, & \textrm{in } B_{dt}(x_0),\\
\varphi_{t,x_0}=\rho_{dt^{\mu}}\!*\! u, & \textrm{on } \partial B_{dt}(x_0),
\end{array}
\right.
\end{equation}
where $d\leq\dist(x_0,\Omega)$, $0<t\leq\frac{1}{2}$, $\mu>1$ and
\begin{align}&f_{1,x_0}(x)=f(x_0)+\langle\nabla f(x_0),x-x_0\rangle,\\
&f_{2,x_0}(x)=f(x_0)+\langle\nabla f(x_0),x-x_0\rangle+\frac{1}{2}(x-x_0)^T\nabla^2f(x_0)(x-x_0).
\end{align}
In fact, $f_{1,x_0}$ and $f_{2,x_0}$ are the first two terms and the first three terms of the Taylor expansion of $f$ respectively.
If $f\in \mathcal{C}^{2,1}(\Omega)$, then we simply use the following auxiliary functions
\begin{equation}\label{f21}
\left\{
\begin{array}{cl}
\det((\varphi_{t,x_0})_{i\bar{j}})=f, & \textrm{in } B_{d}(x_0),\\
\varphi_{t,x_0}=\rho_{dt}\!*\! u, & \textrm{on } \partial B_{d}(x_0),
\end{array}
\right.
\end{equation}
where $d\leq\frac{1}{2}\dist(x_0,\Omega)$ and $0<t\leq\frac{1}{2}$. Here,
$\varphi_{t,x_0}$ defined by (\ref{f21}) is generally not smooth. However, we can obtain necessary estimates on it by considering $\varphi_{t,x_0,\epsilon}\left(0<\epsilon\leq\frac{1}{2}\right)$ satisfying
\begin{equation}\label{f21*}
\left\{
\begin{array}{cl}
\det((\varphi_{t,x_0,\epsilon})_{i\bar{j}})=\rho_{d\epsilon}\!*\!f, & \textrm{in } B_{d}(x_0),\\
\varphi_{t,x_0}=\rho_{dt}\!*\! u, & \textrm{on } \partial B_{d}(x_0),
\end{array}
\right.
\end{equation}
and then letting $\epsilon\rightarrow 0^+$.

\medskip

Using these new auxiliary functions and the argument in the proof of Theorem \ref{mthm}, we can prove the following
theorem. Since the proof is exactly the same as that in Theorem \ref{mthm}, we omit
it.
\begin{theorem} \label{mthm+}
Let $\Omega\subset \mathbb{C}^n$ be a bounded open domain and $u\in \mathcal{C}^{1,\beta}(\Omega)\cap PSH(\Omega)$ be a weak solution of complex Monge-Amp\`ere equation (\ref{CMA}) with $f\in \mathcal{C}(\Omega)$ and $0<\lambda\leq f$ for some constant $\lambda$.\\
(1). If $f\in \mathcal C^{k,\alpha}(\Omega)$ and
\begin{equation}\beta>\beta(n,\tau)=1-\frac{\tau}{n(2+\tau)-2},\end{equation}
where $\alpha \in (0,1]$, $\tau=k+\alpha$ and  $k=1\ \textrm{or}\ 2$, then $u\in \mathcal C^{1,1}(\Omega)$.
Furthermore the $\mathcal C^{1,1}$ norm of $u$ in any relatively compact subset is estimable in terms of $n$, $k+\alpha$, $\beta$, $\lambda$, $\|u\|_{\mathcal C^{1,\beta}(\Omega)}$, $\|f\|_{\mathcal C^{k,\alpha}(\Omega)}$ and the distance of the set to $\partial\Omega$.\\
(2). If $f\in \mathcal C^{2,1}(\Omega)$ and $\beta>1-\frac{1}{n}$, then $u\in \mathcal C^{1,1}(\Omega)$.
Furthermore the $\mathcal C^{1,1}$ norm of $u$ in any relatively compact subset is estimable in terms of $n$, $\beta$, $\lambda$, $\|u\|_{\mathcal C^{1,\beta}(\O)}$, $\|f\|_{\mathcal C^{2,1}(\Omega)}$ and the distance of the set to $\partial\Omega$.
\end{theorem}

\medskip

\hspace{0.3cm}

\section{Appendix}
\setcounter{equation}{0}

The third order estimate or so called the Calabi $\mathcal{C}^3$ estimate plays an important role in the study of higher order regularity of solutions of  complex Monge-Apm\`ere equation. The global case (on compact K\"ahler manifold) is due to Aubin (\cite{Au}) and Yau (\cite{Y}); the local case is due to Riebesehl and Schulz (\cite{SR}). There are many generalizations of the Calabi $\mathcal{C}^3$ estimate (see e.g. \cite{Chau,T,TWY,ZhZh,SW}).

\medskip

We first give an estimate for solutions of  complex Monge-Apm\`ere equation on a ball in $\mathbb{C}^n$. The proof is similar to the proof of Theorem 1.1.(i) of \cite{SW}.
\begin{theorem} \label{calabic3} Let $B_2=B_r(0)\subset \mathbb{C}^n$, and let $u \in PSH(B_2)\cup \mathcal C^5(B_2)$ satisfy
\begin{equation*}\det(u_{i\bar{j}})=e^{\varphi},\end{equation*}
on $B_2$, where $\varphi\in \mathcal C^3(B_2)$.
If there exist constants $\lambda$, $\Lambda$, $M$ and $N$ such that
\begin{equation}\label{S3}\lambda I_n\leq (u_{i\bar{j}})\leq \Lambda I_n,\end{equation}
\begin{equation} -Mr^{-2} I_n\leq(\varphi_{i\bar{j}})\leq Mr^{-2}I_n,\qquad \Big(\sum\limits_{i,j,k}|\varphi_{i\bar{j}k}|^2\Big)^{\frac{1}{2}}\leq Nr^{-3},\end{equation}
on $B_{2}$,
then we have
\begin{equation}\label{6.1}\sum\limits_{i,j,k}|u_{i\bar{j}k}|^2 \leq C(n)\lambda^{-1}\Lambda^3\left((1+M)^{-2}N^2+\lambda^{-1}\Lambda (1+M) \right)r^{-2},\end{equation}
on $B_1=B_{\frac{1}{2}r}(0)$.
\end{theorem}

\medskip

\begin{proof}[Proof.]We denote the K\"ahler metric associated with the K\"ahler form $\sqrt{-1}\partial\bar{\partial}u$ by $g$, and the inverse matrix of $(u_{i\bar{j}})$ by $(u^{i\bar{j}})$.  Similar to \cite{SR}, we introduce the upper indexes, writing in the form
\begin{equation*}v^i=u^{i\bar{j}}v_{\bar{j}},\qquad v^{\bar{i}}=u^{j\bar{i}}v_{j}.\end{equation*}
Set
\begin{equation}\label{S2} S=u^{ik}_j u^j_{ik}.\end{equation}
Denote the standard complex Laplacian by $\Delta$ and the complex Laplacian associated with $g$ by $\Delta_g$. As that in \cite{SR}, we can obtain
\begin{equation}
\begin{aligned}\label{lS}
&\Delta_g \Delta u= u^i_{jk} u^j_{i\bar{k}}+\Delta \varphi,\\
&\Delta_g S =T+\varphi^{lm}_k u^k_{lm}+\varphi^k_{lm} u^{lm}_k-\varphi^p_k\left(u^{lm}_p u^k_{lm}+2u^l_{pm} u^{km}_l\right),
\end{aligned}
\end{equation}
where
\begin{equation}
\begin{aligned}
T=&(u_{ki}^{lm}-u_{pi}^l u_k^{pm}-u_{pi}^m u_k^{lp})(u^{ki}_{lm}-u^{qi}_l u^k_{pm}-u^{pi}_m u^k_{lp})\\
&+(u_{lmi}^k-u_{pi}^k u_{lm}^p)(u^{lmi}_k-u^{pi}_k u^{lm}_p).
\end{aligned}
\end{equation}
Direct calculation shows that
\begin{equation}S_i=u^k_{lm}(u_{ki}^{lm}-u_{pi}^l u_k^{pm}-u_{pi}^m u_m^{kl})+u_k^{lm}(u_{lmi}^k-u_{pi}^k u_{lm}^p).\end{equation}
Diagonalizing $(u_{i\bj})$ at the considered point, it is easy to check that $T\geq 0$. Moreover, applying Cauchy's inequality, we can  obtain
\begin{equation}\label{dS}|dS|_g^2=2S_i S^i\leq 4ST.\end{equation}
Together with the conditions on $(u_{i\bar{j}})$ and $\varphi$, we see that (\ref{lS}) implies
\begin{equation}\label{C300} \begin{split} & \Delta_g \Delta u\geq \lambda S-nMr^{-2},\\
& \Delta_g S \geq T-2\lambda^{-\frac{3}{2}}Nr^{-3}S^{\frac{1}{2}} -3\lambda^{-1}Mr^{-2}S.
\end{split}
\end{equation}
Let $\xi=(r_0^2-|z|^2)^2$ and $\eta =\xi^2$ for any positive $r_0<r$, we define
\begin{equation}G=\eta S+A\lambda^{-1}\Delta u,\end{equation}
on $\bar{B}_{r_0}(0)$. It follows from  (\ref{C300}) and (\ref{dS}) that
\begin{equation}
\begin{aligned}\Delta_g G \geq & (A+\Delta_g \eta-3\lambda^{-1}Mr^{-2}\eta)S+\eta T-2\lambda^{-\frac{3}{2}}Nr^{-3}\eta S^{\frac{1}{2}}\\
& -A\lambda^{-1}nMr^{-2}+\langle d\eta,dS\rangle_g,
\end{aligned}
\end{equation}
and
\begin{equation}|\langle d\eta,dS \rangle_g|\leq |d\eta|_g|dS|_g\leq 2|d\eta|_g(ST)^{\frac{1}{2}}\leq \eta^{-1}|d\eta|_g^2S+\eta T.\end{equation}
By the definition of $\eta$, we have
\begin{equation}\Delta_g \eta=2\xi \Delta_g \xi+|d\xi|_g^2 \geq -2n\lambda^{-1} r_0^2+|d\xi|_g^2,\end{equation}
\begin{equation}\eta^{-1}|d\eta|_g^2=4|d\xi|_g^2\leq 16\lambda^{-1} r_0^2,\end{equation}
and then
\begin{equation}
\Delta_g G \geq (A-C_1 (1+M)\lambda^{-1}r_0^{2})S-C_2 \lambda^{-\frac{3}{2}}Nr_0 S^{\frac{1}{2}}-C_3 A\lambda^{-1}M,
\end{equation}
where $C_1$, $C_2$, $C_3$ and following $C_4$, ... are constants depending only on $n$.
Let $A=2C_1(1+M)\lambda^{-1}r_0^2$, we have
\begin{equation}\label{c3i}
\Delta_g G \geq C_1 (1+M)\lambda^{-1}r_0^{2}S-C_2 \lambda^{-\frac{3}{2}}Nr_0S^{\frac{1}{2}}-C_4 \lambda^{-2}(1+M)M.\end{equation}

\smallskip

By the definition, we can suppose that $G$ achieves its maximum at some interior point, i.e. there exists a point $x_0\in B_{r_0}(0)$ such that
\begin{equation}
G(x_{0})=\max_{x\in \bar{B}_{r_0}(0) } G(x).
\end{equation}
By (\ref{c3i}),  we have
\begin{equation}\label{Sx0}
C_1 (1+M)\lambda^{-1}r_0^{2}S-C_2 \lambda^{-\frac{3}{2}}Nr_0S^{\frac{1}{2}}-C_4 \lambda^{-2}(1+M)M\leq 0,\end{equation}
at the interior maximum point $x_0$, and then
\begin{equation}S(x_0)\leq C_5\lambda^{-1}\left((1+M)^{-2}N^2+M \right)r_0^{-2},\end{equation}
\begin{equation}\label{G}G(x_0)\leq C_6\lambda^{-1}\left((1+M)^{-2}N^2+\lambda^{-1}\Lambda (1+M) \right)r_0^{2}.
\end{equation}
Clearly (\ref{G}) implies
\begin{equation}\max_{x\in \bar{B}_{r_0}(0) }\eta S (x)\leq C_6\lambda^{-1}\left((1+M)^{-2}N^2+\lambda^{-1}\Lambda (1+M) \right)r_0^2.
\end{equation}
Letting $r_0=\frac{3}{4}r$, we obtain
\begin{equation}\max_{x\in B_1}S(x) \leq C_7 \lambda^{-1}\left((1+M)^{-2}N^2+\lambda^{-1}\Lambda (1+M) \right)r^{-2}.
\end{equation}
On the other hand,  (\ref{S2}) and (\ref{S3}) yield
$\sum\limits_{i,j,k}|u_{i\bar{j}k}|^2\leq \Lambda^3 S$,
so  the inequality (\ref{6.1}) follows. This concludes the proof of Theorem \ref{calabic3}.

\end{proof}

\hspace{0.3cm}

Proposition \ref{c3} is a direct corollary of Theorem \ref{calabic3}.
\begin{proof}[Proof of Proposition \ref{c3}]
By the conditions $\Delta u\leq K$ and $f\geq m$, we have
\begin{equation}mK^{-n+1} I_n\leq (u_{i\bar{j}})\leq KI_n,\end{equation}
on $\O$.
Setting $\varphi=\log f$, it is clear that
 $[\log f]^*_{2;\Omega}+[\log f]^*_{3;\Omega}\leq L$ implies
\begin{equation} -Lr^2 I_n\leq(\varphi_{i\bj})\leq Lr^2 I_n,\qquad \Big(\sum\limits_{i,j,k}|\varphi_{i\bar{j}k}|^2\Big)^{\frac{1}{2}}\leq Lr^{-3},\end{equation} on $B_r(x)$, where  $x\in\Omega$ and  $r=\frac{1}{2}d_x$. Applying Theorem \ref{calabic3}, we have
\begin{equation}\Big(\sum\limits_{i,j,k}|u_{i\bar{j}k}|^2\Big)^{\frac{1}{2}}\leq C(n)K^{n+1}m^{-1}(1+L^{\frac{1}{2}})r^{-1},\end{equation}
on $B_{\frac{1}{2}r}(x)$.
The inequality above implies
\begin{equation}d_x|\nabla u_{i\bj}(x)|\leq C(n)K^{n+1}m^{-1}(1+L^{\frac{1}{2}}),
\end{equation}
for any $x\in \O$. This concludes the proof of Proposition \ref{c3}.
\end{proof}

\hspace{0.4cm}

\hspace{0.3cm}

\end{document}